\documentclass[a4paper]{scrartcl}
\usepackage[utf8]{inputenc}
\usepackage{amsmath,amssymb,amsthm,amsfonts}
\usepackage{mathtools,thmtools}
\usepackage[british]{babel}
\usepackage{csquotes}
\usepackage{hyperref}
\usepackage[nameinlink,noabbrev]{cleveref}
\usepackage{enumerate}
\usepackage{xcolor}
\usepackage{framed}

\usepackage[maxbibnames=5,sorting=nyt,sortcites,giveninits=true]{biblatex}
\DeclareNameAlias{sortname}{family-given}
\allowdisplaybreaks

\newcommand{\R}     {\mathbb{R}}

\newcommand{\N}     {\mathbb{N}}
\newcommand{\Z}     {\mathbb{Z}}
\newcommand{\T}     {\mathbb{T}}

\renewcommand{\P}   {\mathbb{P}}
\newcommand{\A}     {\mathcal{A}}

\let\S\relax
\newcommand{\S}     {\mathcal{S}}
\newcommand{\V}     {\mathcal{V}}
\newcommand{\E}     {\mathcal{E}}
\newcommand{\diff}  {\mathop{}\!\mathrm{d}}

\newcommand{\rel}   {\mathrm{rel}}

\renewcommand{\c}   {\mathrm{c}}

\renewcommand{\L}   {\mathcal{L}}
\newcommand{\hatL}  {\hat\L}

\newcommand{\hatLgamma}{\hatL^{(\sigma)}}

\let\phi\varphi

\DeclareMathOperator{\spn}  {span}

\DeclareMathOperator{\gap}  {gap}
\DeclareMathOperator{\spec} {spec}

\DeclareMathOperator{\dom}  {Dom}

\DeclarePairedDelimiterX{\norm}[1]{\lVert}{\rVert}{#1}

\theoremstyle{plain}
\newtheorem{theo}{Theorem}
\newtheorem{lemm}[theo]{Lemma}

\theoremstyle{definition}
\newtheorem{defi}[theo]{Definition}
\newtheorem{exam}[theo]{Example}
\newtheorem{rema}[theo]{Remark}
\newtheorem{rema*}{Remark}
\newtheorem{assu}{Assumption}
\crefname{lemm}{lemma}{lemmas}
\crefname{theo}{theorem}{theorems}
\crefname{assu}{assumption}{assumptions}

\crefname{assualph}{assumption}{assumptions}

\title{Convergence rates of self-repelling diffusions on Riemannian manifolds}

\author{Francis Lörler\thanks{E-Mail: \href{mailto:loerler@uni-bonn.de}{loerler@uni-bonn.de}, ORCID: \href{https://orcid.org/0009-0007-3177-1093}{0009-0007-3177-1093}}\medskip\\Institute for Applied Mathematics, University of Bonn.}

\begin{document}

\maketitle

\begin{abstract}
    We study a class of self-repelling diffusions on compact Riemannian manifolds whose drift is the gradient of a potential accumulated along their trajectory. When the interaction potential admits a suitable spectral decomposition, the dynamics and its environment are equivalent to a finite-dimensional degenerate diffusion. We show that this diffusion is a second-order lift of an Ornstein-Uhlenbeck process whose invariant law corresponds to the Gaussian invariant measure of the environment, and immediately obtain a general upper bound on the rate of convergence to stationarity using the framework of second-order lifts. Furthermore, using a flow Poincar\'e inequality, we develop lower bounds on the convergence rate. We show that, in the periodic case, these lower bounds improve upon those of Benaïm and Gauthier \cite{BenaimGauthier2016Selfrepelling}, and even match the order of the upper bound in some cases.

    \begin{samepage}
        \par\vspace{0.5\baselineskip}
    \end{samepage}
\end{abstract}

\section{Introduction}

We consider a class of self-interacting diffusions $(X_t)_{t\geq 0}$ on compact, connected and oriented smooth Riemannian manifolds $M$. Given a smooth and symmetric interaction potential $V\colon M\times M\to\R$, the process $(X_t)_{t\geq 0}$ solves the stochastic differential equation
\begin{equation}\label{eq:SDE}
    \diff X_t\ =\ -\nabla_M V_t(X_t)\diff t\ +\ \sigma\circ\diff B_t\,,\qquad X_0\sim\mu_0\,,
\end{equation}
where $\sigma\geq 0$ is the diffusivity, $(B_t)_{t\geq 0}$ is a Brownian motion on $M$, and $\nabla_M$ denotes the Riemannian gradient on $M$. The drift is the gradient of the accumulated potential along the trajectory up to time $t$, i.e.\ 
\begin{equation}\label{eq:Vt}
    V_t(x) = \int_0^tV(X_s,x)\diff s = \int_MV(y,x)\mu_t(\diff y)\,,\quad\text{where }\mu_t = \int_0^t\delta_{X_s}\diff s\,,
\end{equation}
and the initial law $\mu_0$ is a probability measure on $M$. The process $(\nabla_MV_t)_{t\geq 0}$ is also known as the environment process.

Such self-repelling processes were originally motivated by and introduced as models for polymer growth, notably by Norris, Rogers and Williams \cite{NorrisRogersWilliams1987Selfavoiding,NorrisRogersWilliams1985Excluded}, where the asymptotic growth of the polymer is of particular interest. In case $M=\R^d$, the asymptotic behaviour in the form of strong laws of large numbers has been studied in the works \cite{DurrettRogers1992Asymptotic,CranstonMountford1996BrownianPolymer,MountfordTarres2008BrownianPolymer}. In case that $M$ is compact, as considered here, one expects the law of $X_t$ to converge to an invariant probability measure. Indeed, under a variety of different assumptions, such a convergence has been shown using the normalised occupation time measure $\frac{1}{t}\mu_t$ in place of $\mu_t$ in the series of works \cite{BenaimLedouxRaimond2002Selfinteracting,BenaimOlivier2003SelfinteractingII,BenaimOlivier2005SelfinteractingIII}, and with the non-normalised occupation time measure in \cite{BenaimGauthier2016Selfrepelling,BCG2015selfrepelling}. 

Self-interacting processes such as the process $(X_t)_{t\geq 0}$ are not Markovian since they depend on their own history through their occupation measure, rendering their study notoriously difficult. However, the joint process $(V_t,X_t)_{t\geq 0}$ is a Markov process, and it has been shown by Tarr{\`e}s, T{\'o}th and Valk{\'o} \cite{TarresTothValko2012Diffusivity} that, under suitable assumptions on the interaction potential $V$, the combined process $(x\mapsto\nabla_M V_t(X_t+x))_{t\geq 0}$, known as the environment as seen from the particle, admits an ergodic invariant Gaussian measure. Following on this work, Benaïm and Gauthier \cite{BenaimGauthier2016Selfrepelling} considered the case that the interaction potential admits a finite diagonal decomposition in terms of eigenfunctions of the Laplace-Beltrami operator $\Delta_M$, see \Cref{assu:V} below. In this case, the process is \emph{self-repelling}, and \eqref{eq:SDE} is equivalent to a finite-dimensional system of stochastic differential equations whose solution is a degenerate, non-reversible diffusion process on $\R^n\times M$, with invariant probability measure given by the product of a Gaussian and the uniform measure on $M$, as shown by \cite{BenaimGauthier2016Selfrepelling}.

Even though the invariant probability measure admits a simple and explicit description, obtaining rates of convergence of the law of the process to stationarity is challenging due to its non-reversibility and degeneracy. Classical analytical techniques based on Poincar\'e inequalities and spectral estimates cannot be applied in the non-reversible, degenerate setting, necessitating the development of completely new techniques in this area of research that has become known as hypocoercivity \cite{Villani2009Hypocoercivity}. In the seminal work \cite{DMS2015Hypocoercivity}, Dolbeault, Mouhot and Schmeiser introduced a framework for the study of a large class of hypocoercive kinetic equations. This framework was further developed by \cite{GrothausStilgenbauer2014Hypocoercivity} and then applied to self-repelling diffusions on Riemannian manifolds by Benaïm and Gauthier \cite{BenaimGauthier2016Selfrepelling}. The DMS approach of \cite{DMS2015Hypocoercivity}, however, fails to provide rates of convergence to stationarity that are sharp in terms of the model parameters in many settings. Indeed, for Langevin dynamics (or the associated kinetic Fokker-Planck equation), a prototypical non-reversible degenerate dynamics, sharp rates of convergence have only recently been established in \cite{Cao2019Langevin} using a variational approach to hypocoercivity based on space-time Poincar\'e inequalities, introduced by Albritton, Armstrong, Mourrat and Novack \cite{Albritton2021Variational}. This approach was further developed and made broadly applicable using the framework of second-order lifts \cite{EberleLoerler2024Lifts,EGHLM2025Convergence} that allows to prove such space-time Poincar\'e inequalities by relating the non-reversible dynamics to simpler, reversible diffusions. 

In this work, we identify the self-repelling diffusion and its environment as a second-order lift of an Ornstein-Uhlenbeck process with invariant law corresponding to the Gaussian measure that is the invariant law of the environment process. This provides a new perspective on these processes and allows us to apply the framework of non-reversible lifts of Markov processes \cite{EberleLoerler2024Lifts, EGHLM2025Convergence} to obtain upper and lower bounds on their rate of convergence to equilibrium. Our upper bounds improve on those of \cite{BenaimGauthier2016Selfrepelling} obtained using the DMS approach, as demonstrated by an example on the torus, and are even sharp in some settings.

\medskip

We begin by introducing self-interacting diffusions and the equivalent finite-dimensional system in \Cref{sec:setting} and show in \Cref{sec:lifts} that, under \Cref{assu:V}, self-repelling diffusions on Riemannian manifolds are second-order lifts of an Ornstein-Uhlenbeck process on $\R^n$ with Gaussian invariant probability measure. The resulting upper and lower bounds on the $L^2$-relaxation time of the associated transition semigroups are given in \Cref{ssec:relaxationtimes}. We compare our upper bounds to those of \cite{BenaimGauthier2016Selfrepelling} in the periodic case in \Cref{ssec:periodic}. Finally, we conclude by giving the proof of the upper bounds on the relaxation time in \Cref{sec:convergence}.

\section{Self-interacting diffusions on Riemannian manifolds}\label{sec:setting}

Consider a compact, connected and oriented smooth Riemannian manifold $(M,g)$ without boundary and a smooth, symmetric function $V\colon M\times M\to\R$. A \emph{self-interacting diffusion with potential $V$} is the $M$-valued continuous-time stochastic process $(X_t)_{t\geq 0}$ on some probability space $(\Omega,\A,\P)$ that solves the stochastic differential equation
\begin{equation*}
    \diff X_t\ =\ -\nabla_M V_t(X_t)\diff t\ +\ \sigma\circ\diff B_t\,,\qquad X_0\sim\mu_0\,,
\end{equation*}
Here $\sigma\geq 0$ is the diffusivity and $(B_t)_{t\geq 0}$ is a Brownian motion on $M$, the drift is the gradient of the accumulated potential as given in \eqref{eq:Vt}, and the initial law $\mu_0$ is a probability measure on $M$. In the following, we let $\nu$ denote the Riemannian volume measure on $M$. We will work under the following assumption.
\begin{assu}\label{assu:V}
    $V$ takes the form
    \begin{equation}\label{eq:VinONS}
        V(x,y) = \sum_{k=1}^na_ke_k(x)e_k(y),
    \end{equation}
    where $(e_1,\dots,e_n)$ is a system of $L^2(M,\nu)$-orthonormal eigenfunctions of the Laplace-Beltrami operator $\Delta_M$ on $M$, with associated non-positive eigenvalues $\lambda_1,\dots,\lambda_n$, and $a_1,\dots,a_n\in [0,\infty)$.
\end{assu}
\Cref{assu:V} ensures that $V$ is a symmetric and positive semi-definite kernel, i.e.\ 
\begin{equation*}
    \int_{M}\int_MV(x,y)f(x)f(y)\ \nu(\diff x)\nu(\diff y)\geq 0
\end{equation*}
for all $f\in L^2(M,\nu)$, so that the process $(X_t)_{t\geq 0}$ is \emph{self-repelling}. Moreover, as shown in \cite{BenaimGauthier2016Selfrepelling}, the process $(\nabla_M V_t,X_t)_{t\geq 0}$ admits an explicit invariant probability measure supported on
\begin{equation*}
    \spn\bigl\{\nabla_M e_k: k\in\{1,\dots,n\}\bigr\}\times M\ \subseteq\ \mathfrak{X}(M)\times M\,,
\end{equation*}
where $\mathfrak{X}(M)$ denotes the set of all smooth vector fields on $M$, see \Cref{ssec:finitedim} below.

\begin{rema}\label{rem:localtime}
    If $M=\T^d$, $\mu_t$ has a density $L_t$ with respect to the volume measure and $V(x,y) = W(x-y)$ for some symmetric function $W\colon \T^d\to\R$, the accumulated potential simplifies to
    \begin{equation*}
        V_t(x) = \int_{\T^d}W(y-x)L_t(y)\,\nu(\diff y) = (W*L_t)(x).
    \end{equation*}
    In this case, $(X_t)$ is a diffusion driven by the mollified negative gradient of its own local time density $L_t$. In the non-compact case $M=\R^d$, corresponding processes are known as self-repelling Brownian polymers \cite{DurrettRogers1992Asymptotic}, which have recently been constructed even in the singular case $W = \delta_0$ \cite{GilesGraefner2025SRBP1d}. They are phenomenologically closely related to the true self-repelling motion introduced by T\'oth and Werner \cite{TothWerner1998TSRM}.
\end{rema}

\subsection{Reduction to a finite-dimensional system}\label{ssec:finitedim}

In this section, we follow the construction in \cite[Section 2]{BenaimGauthier2016Selfrepelling} that shows how the structure of $V$ given in \Cref{assu:V} allows to reduce the self-repelling diffusion and its environment to a system of stochastic differential equations in $\R^n\times M$.

\Cref{assu:V} allows us to decompose
\begin{equation*}
    \nabla_M V_t(x)\ =\ \sum_{k=1}^na_k\nabla_M e_k(x)U^{(k)}_t\qquad\text{with}\quad U^{(k)}_t=\int_0^te_k(X_s)\diff s\,.
\end{equation*}
The self-repelling diffusion \eqref{eq:SDE} and its environment are hence equivalent to the system
\begin{align*}
    \diff U^{(k)}_t\ &=\ e_k(X_t)\diff t\,,\qquad k=1,\dots,n\\
    \diff X_t\ &=\ -\sum_{k=1}^na_k\nabla_M e_k(X_t)U^{(k)}_t\diff t\ +\ \sigma\circ\diff B_t
\end{align*}
of stochastic differential equations on $\R^n\times M$, where $\circ\diff B_t$ denotes Stratonovich integration. Let $U_t = (U_t^{(k)})_{k=1,\dots,n}$ and $e(x) = (e_k(x))_{k=1,\dots,n}$,
and introduce the potential
\begin{equation*}
    \Phi(u)\ =\ \frac{1}{2}\sum_{k=1}^na_k|\lambda_k|u_k^2\,,\qquad u\in\R^n.
\end{equation*}
Then this system can equivalently be written as
\begin{equation}
    \begin{aligned}\label{eq:XtUtSDE}
        \diff U_t\ &=\ e(X_t)\diff t\\
        \diff X_t\ &=\ -\nabla\Phi(U_t)^\top\Lambda^{-1}\nabla_M e(X_t)\diff t\ +\ \sigma\circ\diff B_t\,,\\
    \end{aligned}
\end{equation}
where  $\Lambda = \mathrm{diag}(|\lambda_1|,\dots,|\lambda_n|)\in\R^{n\times n}$ and $\nabla_M e(x) = (\nabla_M e_k(x))_{k=1,\dots,n}$. The drift is thus to be understood as
\begin{equation*}
    \nabla \Phi(U_t)^\top\Lambda^{-1}\nabla_M e(X_t) = \sum_{k=1}^n\partial_k\Phi(U_t)|\lambda_k|^{-1}\nabla_M e_k(X_t)\,.
\end{equation*}
Note that, to avoid confusion, we use $\nabla$ for the Euclidean gradient in $\R^n$ and $\nabla_M$ for the gradient on $M$.

We summarise some properties of the system \eqref{eq:XtUtSDE} that are given in \cite{BenaimGauthier2016Selfrepelling} in the following lemma.

\begin{lemm}\label{lem:invariant}
    \begin{enumerate}[(i)]
        \item There exists a unique global strong solution $(U_t,X_t)_{t\geq 0}$ to \eqref{eq:XtUtSDE} for any initial condition $(U_0,X_0)=(u,x)\in\R^d\times M$.

        \item The probability measure $$\hat\mu = \mu\otimes\kappa$$ with
        \begin{equation*}
            \mu(\diff u) = \frac{1}{Z}\exp(-\Phi(u))\diff u\qquad\text{and}\qquad\kappa(\diff x) = \frac{1}{\nu(M)}\nu(\diff x),
        \end{equation*}
        where $Z$ is a normalising constant, is invariant for this process. If $\sigma>0$, it is the unique invariant probability measure.
    \end{enumerate}
\end{lemm}
\begin{proof}
    This is the content of \cite[Proposition 1]{BenaimGauthier2016Selfrepelling} and \cite[Theorem 5]{BenaimGauthier2016Selfrepelling}.
\end{proof}

Note that $\Phi$ is quadratic, so that 
\begin{equation*}
    \mu\ =\ \bigotimes_{k=1}^n\mathcal{N}\big(0,\tfrac{1}{a_k|\lambda_k|}\big)
\end{equation*}
is a Gaussian distribution. If $\sigma=0$ the invariant probability measure is possibly not unique, as there may exist infinitely many ergodic probability measures; see \cite[Theorem 3]{BenaimGauthier2016Selfrepelling} for an explicit example on the torus.

By \Cref{lem:invariant}, the process $(U_t,X_t)$ induces a strongly continuous contraction semigroup $(\hat P_t^{(\sigma)})_{t\geq 0}$ on $L^2(\hat\mu)$, and we denote its generator by $(\hatLgamma,\dom(\hatLgamma))$. By \cite[Propositions 2 and 3]{BenaimGauthier2016Selfrepelling}, the set $C_\c^2(\R^n\times M)$ of compactly supported, twice differentiable functions on $\R^n\times M$ is a core for the generator, and on these functions it acts as
\begin{equation}\label{eq:generator}
    \begin{aligned}
        \hatLgamma f(u,x)
        &\ =\ \sum_{k=1}^n\big(e_k(x)\partial_{u_k}f(u,x) - a_ku_k\langle\nabla_M e_k(x),\nabla_Mf(u,x)\rangle\big) + \frac{\sigma^2}{2}\Delta_Mf(u,x)\\
        &\ =\ e(x)^\top\nabla f(u,x) - \big\langle\nabla\Phi(u)^\top\Lambda^{-1}\nabla_M e(x),\nabla_Mf(u,x)\big\rangle + \frac{\sigma^2}{2}\Delta_Mf(u,x)\,.
    \end{aligned}
\end{equation}

\begin{exam}\label{ex:torus}
    Consider $M = \T_L = \R/(2\pi L\Z)$ the circle with circumference $2\pi L$ and
    \begin{equation*}
        W(z) = \frac{1}{\pi L}\sum_{k=1}^na_k\cos(kz/L)\,,\qquad z\in\T_L\,,
    \end{equation*}
    where $a_1,\dots,a_n\in[0,\infty)$.
    Then setting $V(x,y) = W(x-y)$ yields
    \begin{equation*}
        V(x,y) = \frac{1}{\pi L}\sum_{k=1}^na_k\cos(kx/L)\cos(ky/L)+a_k\sin(kx/L)\sin(ky/L)\,,
    \end{equation*}
    which is of the form \eqref{eq:VinONS} with the orthonormal functions $e_k(x) = (\pi L)^{-1/2}\cos(kx/L)$ and $f_k(x) =(\pi L)^{-1/2}\sin(kx/L)$. The associated eigenvalues are $\lambda_k = -\frac{k^2}{L^2}$. The potential $\Phi$ then becomes
    \begin{equation*}
        \Phi(u,v) = \frac{1}{2L^2}\sum_{k=1}^na_kk^2(u_k^2+v_k^2)\,,\qquad u,v\in\R^n\,,
    \end{equation*}
    where $u_k$ and $v_k$ correspond to the eigenfunctions $e_k$ and $f_k$. We will return to this example in \Cref{ssec:periodic}.
\end{exam}

\section{Relaxation times of self-interacting diffusions}\label{sec:lifts}

The process $(U_t,X_t)_{t\geq 0}$ is non-reversible and subject to degenerate noise acting only on $X_t$. Therefore, studying the rate of convergence of the law of $(U_t,V_t)$ to its invariant distribution $\hat\mu$ is a challenging task and falls within the field of hypocoercivity. In this section, we will apply the recently developed framework of second-order lifts and flow Poincar\'e inequalities to obtain upper and lower bounds on this rate of convergence in the form of bounds on the $L^2$-relaxation time.

\smallskip

We begin by recalling the concept of second-order lifts in \Cref{ssec:deflifts}, and show how self-repelling diffusions and their environment can be seen as lifts of an Ornstein-Uhlenbeck process with invariant measure $\mu$ in \Cref{ssec:selfrepellinglift}. We then apply the framework of \cite{EGHLM2025Convergence} to obtain upper and lower bounds on the $L^2$-relaxation time of the transition semigroups associated to the self-repelling diffusions in \Cref{ssec:relaxationtimes}. These are compared to those obtained by \cite{BenaimGauthier2016Selfrepelling} in the periodic case in \Cref{ssec:periodic}.

\subsection{Second-order lifts}\label{ssec:deflifts}

Inspired by the notion of lifts of Markov chains in discrete time \cite{Diaconis2000Lift,Chen1999Lift}, second-order lifts of reversible diffusions were introduced in \cite{EberleLoerler2024Lifts} as a related counterpart in continuous time and space. They provide a broadly applicable and systematic approach developed in \cite{EberleLoerler2024Lifts,EberleLoerler2024Divergence,EGHLM2025Convergence} to proving convergence to equilibrium for many non-reversible dynamics by relating them to simpler, reversible diffusions.

Intuitively, a time-homogeneous Markov process $(\hat Z_t)_{t\geq 0}$ on a product space $\S\times\V$ with invariant probability measure $\hat\mu = \mu\otimes\kappa$ is a second-order lift of another Markov process $(Z_t)_{t\geq 0}$ with state space $\S$ and invariant probability measure $\mu$ if their associated transition semigroups $(\hat P_t)$ and $(P_t)$ in $L^2(\hat\mu)$ and $L^2(\mu)$, respectively, satisfy
\begin{equation*}
    \int_\V\hat P_t(f\circ\pi)(x,v)\kappa(\diff v) = (P_{t^2}f)(x) + o(t^2)\qquad\text{as }t\downarrow0
\end{equation*}
for functions $f\colon\S\to\R$ in the domain of the generator $\L$ of $(P_t)$. Here $\pi\colon\S\times\V\to\S$ denotes the canonical projection $\pi(x,v)=x$. Considering a second-order Taylor expansion around $t=0$ leads to the precise definition below, see \cite{EberleLoerler2024Lifts} for more intuition behind this property and some examples. Let $(\L,\dom(\L))$ and $(\hatL,\dom(\hatL))$ denote the generators of $(P_t)$ and $(\hat P_t)$ in $L^2(\mu)$ and $L^2(\hat\mu)$, respectively. We denote the Dirichlet form associated to $\L$ by $\E$, it is the extension of 
\begin{equation}\label{eq:DiriForm}
    \E(f,g)=-\int_\S f\L g\diff\mu
\end{equation}
to a closed symmetric bilinear form with domain $\dom(\E)$ given by the closure of $\dom(\L)$ with respect to the norm $\|f\|_{L^2(\mu)} + \E(f,f)^{1/2}$.

\begin{defi}[Second-order lifts \cite{EberleLoerler2024Lifts}]\label{def:lift}
    The process $(\hat Z_t)$ is a \emph{second-order lift} of $(Z_t)$ if there exists a core $C$ of $(\L,\dom(\L))$ such that
    \begin{align}\label{eq:deflift0}
        f \circ \pi \in \dom(\hatL) \text{ for all } f \in C
    \end{align}
    and for all $f,g\in C$ we have
    \begin{equation}\label{eq:deflift1}
         \big< \hatL (f \circ \pi), g \circ \pi \big>_{L^2(\hat\mu)} = 0,
    \end{equation}
    and
    \begin{equation}\label{eq:deflift2}
        \frac{1}{2}\big< \hatL (f \circ \pi), \hatL(g \circ \pi) \big>_{L^2(\hat\mu)} = -\left< f, \L g\right>_{L^2(\mu)} = \mathcal E(f, g).
    \end{equation}
    Conversely, we refer to the process $(Z_t)$ as the \emph{collapse} of $(\hat Z_t)$.
\end{defi}

We also say that $(\hatL,\dom(\hatL))$ or $(\hat P_t)$ is a second-order lift of $(\L,\dom(\L))$ or $(P_t)$, respectively. In the rest of this paper, we simply refer to second-order lifts as lifts.

\subsection{Self-interacting diffusions as second-order lifts}\label{ssec:selfrepellinglift}

In this section, we show that the process $(U_t,X_t)_{t\geq 0}$ that is equivalent to the self-repelling diffusion and its environment is a second-order lift of an Ornstein-Uhlenbeck process on $\R^n$ for any choice of diffusivity.
To this end, we consider the transition semigroup $(\hat P_t^{(\sigma)})_{t\geq 0}$ of $(U_t,X_t)_{t\geq 0}$ with diffusivity $\sigma$ on $L^2(\hat\mu)$, and recall that,
for compactly supported functions $f\in C_\c^2(\R^n\times M)$, the associated generator is given by
\begin{align*}
    \hatLgamma f(u,x)
    &=\sum_{k=1}^n\big(e_k(x)\partial_{u_k}f(u,x) - a_ku_k\langle\nabla_Me_k(x),\nabla_Mf(u,x)\rangle\big) + \frac{\sigma^2}{2}\Delta_Mf(u,x)\\
    &=e(x)^\top\nabla f(u,x) - \big\langle\nabla\Phi(u)^\top\Lambda^{-1}\langle\nabla_Me(x),\nabla_Mf(u,x)\big\rangle + \frac{\sigma^2}{2}\Delta_Mf(u,x)\,,
\end{align*}
where $\Lambda = \mathrm{diag}(|\lambda_1|,\dots,|\lambda_n|)$.
In the following, we let $\pi(u,x) = u$,
\begin{equation*}
    (\hatL,\dom(\hatL)) = (\hatL^{(0)},\dom(\hatL^{(0)}) \quad\text{and}\quad (\hat P_t)_{t\geq 0} = (\hat P_t^{(0)})_{t\geq0}\,.
\end{equation*}
We will show that the self-repelling diffusion $(U_t,X_t)_{t\geq 0}$ is a second-order lift of the reversible diffusion process $(Z_t)_{t\geq 0}$ on $\R^n$ that solves the stochastic differential equation
\begin{equation}\label{eq:OU}
    \diff Z_t\ =\ -\frac{1}{2\nu(M)}\nabla\Phi(Z_t)\diff t + \frac{1}{\sqrt{\nu(M)}}\diff W_t\,,
\end{equation}
where $(W_t)_{t\geq 0}$ is an $n$-dimensional standard Brownian motion. 
We denote the transition semigroup of $(Z_t)_{t\geq 0}$ on $L^2(\mu)$ by $(P_t)_{t\geq 0}$. Its generator $(\L,\dom(\L))$ is given by
\begin{equation}\label{eq:OUgenerator}
    \L f\ =\ -\frac{1}{2\nu(M)}\nabla^*\nabla f\ =\ -\frac{1}{2\nu(M)}\nabla\Phi\cdot\nabla f + \frac{1}{2\nu(M)}\Delta f,
\end{equation}
where the adjoint is in $L^2(\mu)$, and $\dom(\hatL) = H^{2,2}(\R^n,\mu)$. Note that \eqref{eq:OU} is an Ornstein-Uhlenbeck process on $\R^n$ with invariant probability measure $\mu$, slowed down by a factor $\frac{1}{2\nu(M)}$. In particular, the spectral gap of the generator in $L^2(\mu)$ is
\begin{equation}\label{eq:OUgap}
    \gap(\L) = \frac{1}{2\nu(M)}\min\bigl\{a_k|\lambda_k| : k\in\{1,\dots,n\}\bigr\}\,,
\end{equation}
see \cite{Metafune2002Spectrum}.

\begin{theo}\label{thm:lift}
    The semigroup $(\hat P_t^{(\sigma)})_{t\geq 0}$ is a second-order lift of $(P_t)_{t\geq 0}$ for any choice of diffusivity $\sigma\geq 0$.
\end{theo}
\begin{proof}
    It suffices to prove the claim for $\sigma = 0$. Firstly, $C_\c^\infty(\R^n)$ forms a core for $(\L,\dom(\L))$ \cite{Wielens1985Selfadjointness}. Hence let $g\in C_\c^\infty(\R^n)$. This implies $g\circ\pi\in\dom(\hatL)$ and
    \begin{equation*}
        \int_M\hatL(g\circ\pi)(u,x)\kappa(\diff x) = \int_Me(x)^\top\nabla f(u)\kappa(\diff x) = 0
    \end{equation*}
    since $M$ has no boundary. Furthermore, since the $e_k$ are orthonormal with respect to $\nu$, 
    \begin{equation*}
        \frac{1}{2}\int_M(\hatL(g\circ\pi)(u,x))^2\kappa(\diff x) = \frac{1}{2}\int_Me(x)^\top\nabla g(u)\nabla g(u)^\top e(x)\kappa(\diff x) = \frac{1}{2\nu(M)}|\nabla g(u)|^2,
    \end{equation*}
    so that 
    \begin{equation*}
        \frac{1}{2}\int_{\R^n\times M}(\hatL(g\circ\pi))^2\diff\hat\mu = \frac{1}{2\nu(M)}\int_{\R^n}|\nabla g|^2\diff\mu = -\int _{\R^n}g\L g\diff\mu\,,
    \end{equation*}
    and \eqref{eq:deflift2} follows by polarisation.
\end{proof}

\subsection{Bounds for relaxation times}\label{ssec:relaxationtimes}

We first note that the representation of the self-repelling diffusions as a second-order lift of the Ornstein-Uhlenbeck process \eqref{eq:OU} directly yields a lower bound on the $L^2$-relaxation time
$$t_{\rel}(\hat P)\ =\ \inf\left\{ t\ge 0 :\| \hat P_tf\|_{L^2(\hat\mu )}\le e^{-1}\| f\|_{L^2(\hat\mu )}\text{ for all }f\in L_0^2(\hat\mu )\,\right\} $$
of the semigroup $(\hat P_t)$.
Here $L_0^2(\hat\mu )$ denotes the orthogonal complements of the constant functions in $L^2(\hat \mu )$. Let us stress that, since $(\hat P_t^{(\sigma)})_{t\geq 0}$ are the transition semigroups of the processes $(U_t,X_t)_{t\geq 0}$ with diffusivity $\sigma$, all relaxation time bounds concern the joint relaxation time of the environment parametrised by $U_t$ and the self-repelling diffusion $X_t$ itself.

\begin{theo}[Lower bound on relaxation time]\label{thm:lowerbound}
For the transition semigroup $(\hat P^{(\sigma)}_t)_{t\geq 0}$ associated to the self-repelling diffusion with diffusivity $\sigma\geq 0$,
   $$ t_{\rel}(\hat P^{(\sigma)})\ \geq\ \left(\frac{\nu(M)}{4\min\bigl\{a_k|\lambda_k| : k\in\{1,\dots,n\}\bigr\}}\right)^{1/2}\,.$$
\end{theo}
\begin{proof}
    This is a consequence of \Cref{thm:lift} above and \cite[Theorem 11]{EberleLoerler2024Lifts} which states that $t_{\rel}(\hat P )\ge \sqrt{t_{\rel}(P)}/(2\sqrt 2)$ holds whenever $P=(P_t)_{t\ge 0}$ and $\hat P=(\hat P_t)_{t\ge 0}$ are the transition semigroups of a Markov process and an arbitrary second order lift. By reversibility, the relaxation time of the Ornstein-Uhlenbeck process \eqref{eq:OU} is the inverse of the spectral gap of its generator $\L$.
    Therefore, \eqref{eq:OUgap} yields the claim.
\end{proof}

Obtaining quantitative upper bounds on the relaxation time is a lot more involved. Second-order lifts provide a framework developed in \cite{EberleLoerler2024Lifts,EberleLoerler2024Divergence,EGHLM2025Convergence} for a systematic approach to this problem, based on a variational approach to hypocoercivity first introduced in \cite{Albritton2021Variational,Cao2019Langevin}. Quantitative rates of convergence are established using a flow Poincaré inequality relating the time-averaged $L^2$-norm to the time-averaged energy or Dirichlet form. 
Previous approaches to proving convergence for non-reversible dynamics relied on tricks such as considering Sobolev norms \cite{Villani2009Hypocoercivity}, adding a mixing term to create an equivalent norm in which dissipation takes place \cite{DMS2015Hypocoercivity}, or specifically tailored couplings and twisted metrics \cite{Eberle2019Langevin}.
The strength of the approach using second-order lifts lies in the more intrinsic nature of the time-averaged $L^2$-norms as opposed to such artificially introduced terms \cite{BLW24Hypocoercivity}. We summarise the resulting bounds for the convergence of self-repelling diffusions in the following, and defer their proofs and an overview of the framework to \Cref{sec:convergence}.
Since the approach \cite{DMS2015Hypocoercivity} developed by Dolbeault, Mouhot and Schmeiser was already applied to self-repelling diffusions by Benaïm and Gauthier in \cite{BenaimGauthier2016Selfrepelling}, we compare the result of \Cref{thm:upperbound} to theirs in \Cref{ssec:periodic}.

We let
\begin{equation*}
    \chi_{ijkl} = \int_Me_ie_je_ke_l\diff\nu\quad \text{and}\quad\tilde\chi_{ijkl} = \int_M\langle\nabla_Me_i,\nabla_Me_j\rangle\langle\nabla_Me_k,\nabla_Me_l\rangle|\lambda_i|^{-1}|\lambda_k|^{-1}\diff\nu\,,
\end{equation*}
as well as
\begin{equation*}
    \chi = \Big(\sum_{i,j,k,l=1}^n\chi_{ijkl}^2\Big)^{1/2}\qquad\text{and}\qquad\tilde\chi = \Big(\sum_{i,j,k,l=1}^n\tilde\chi_{ijkl}^2\Big)^{1/2}\,.
\end{equation*}

\begin{theo}[Upper bound on the relaxation time]\label{thm:upperbound}
    The transition semigroup $(\hat P_t^{(\sigma)})_{t\geq0}$ satisfies
    \begin{equation}
        \norm{\hat P_t^{(\sigma)}}_{L^2(\hat\mu)}\ \leq\ e^{-\nu(t-T)}\norm{f}_{L^2(\hat\mu)}\qquad\text{for all }t\geq 0\text{ and }f\in L_0^2(\hat\mu) 
    \end{equation}
    with $T \in O(m^{-1/2})$ and inverse convergence rate
    \begin{equation*}
        \nu^{-1} \in O\left(\sigma^2C_2^2 + \sigma^{-2}\Big(C_1^2 + \frac{1}{\eta}\Big)\right)\,,
    \end{equation*}
    where
    \begin{align*}
        C_1^2 &= 8\nu(M)\left(\chi + 4\tilde\chi + \frac{2\tilde\chi\sum_{k=1}^na_k|\lambda_k|}{\min\bigl\{a_k|\lambda_k| : k\in\{1,\dots,n\}\bigr\}}\right)\,,\\
        C_2^2 &= \frac{4\nu(M)}{\min\bigl\{a_k:k\in\{1,\dots,n\}\bigr\}}\,,\\
        \eta &= \inf\spec(-\Delta_M)\setminus\{0\}\,,\\
        m &= \frac{\min\bigl\{a_k|\lambda_k| : k\in\{1,\dots,n\}\bigr\}}{2\nu(M)}\,.
    \end{align*}
    In particular, there exists a universal constant $C>0$ such that
    \begin{equation*}
        t_\rel(\hat P^{(\sigma)})\ \leq\ C\, \Big(\sigma^2C_2^2 + \sigma^{-2}\Big(C_1^2 + \frac{1}{\eta}\Big)\Big)
    \end{equation*}
    for any choice of diffusivity $\sigma>0$.
\end{theo}

Note that compactness of $M$ ensures $\eta>0$. Furthermore, if the Ricci curvature of $M$ is non-negative, the spectral gap satisfies $\eta\geq\frac{\pi^2}{\mathrm{diam}(M)^2}$, see \cite[Theorem 5.5]{Li2012GeometricAnalysis}.

\subsection{The periodic case}\label{ssec:periodic}

As a concrete example, we return to the setting of \Cref{ex:torus} where $M = \T_L = \R/(2\pi L\Z)$ is the circle with circumference $2\pi L$ and choose
\begin{equation}\label{eq:interactionW}
    W(z) = \frac{1}{\pi L}\sum_{k=1}^na_k\cos(kz/L)\,,\qquad z\in\T_L\,.
\end{equation}
Recall that setting $V(x,y) = W(x-y)$ yields
\begin{equation*}
    V(x,y) = \frac{1}{\pi L}\sum_{k=1}^na_k\cos(kx/L)\cos(ky/L)+a_k\sin(kx/L)\sin(ky/L),
\end{equation*}
which is of the form \eqref{eq:VinONS} with the orthonormal functions $e_k(x) = (\pi L)^{-1/2}\cos(kx/L)$ and $f_k(x) =(\pi L)^{-1/2}\sin(kx/L)$. The associated eigenvalues are $\lambda_k = -\frac{k^2}{L^2}$ and the potential $\Phi$ becomes
\begin{equation*}
    \Phi(u,v) = \frac{1}{2L^2}\sum_{k=1}^na_kk^2(u_k^2+v_k^2)\,,\qquad u,v\in\R^n\,,
\end{equation*}
where $u_k$ and $v_k$ correspond to the eigenfunctions $e_k$ and $f_k$. We then have 
\begin{equation*}
    \nu(M) = 2\pi L\qquad\text{and}\qquad\eta = \frac{1}{L^2}\,.
\end{equation*}
The lower bound in \Cref{thm:lowerbound} yields
\begin{equation}\label{eq:lowerboundperiodic}
    t_\rel(\hat P^{(\sigma)})\ \geq\ \left(\frac{\pi L^3}{2\min\bigl\{a_kk^2 : k\in\{1,\dots,n\}\bigr\}}\right)^{1/2}\,. 
\end{equation}
The dependence of the relaxation time on the size parameter $L$ in the limit $L\to\infty$ is of particular interest, as it gives information on the rate at which the process explores its ambient space.

\begin{rema}
    As explained in \Cref{rem:localtime}, in this setting $V_t(x) = (W*L_t)(x)$, where $L_t$ is the local time density of $(X_t)$ at time $t$. The finite-dimensional system \eqref{eq:XtUtSDE} introduced above is hence obtained as the gradient of an approximation of the local time in Fourier domain. Instead considering spatial discretisations of the local time leads to self-repelling random walks on the discrete circle. These can be treated using the framework of second-order lifts in a similar fashion, see \cite{EberleLoerler2025SRW_ECMC}.
\end{rema}

\begin{exam}\label{ex:torus1}
    As a first simple case, consider
    \begin{equation*}
        V(x,y) = a\cos(k(x-y)/L) = \frac{1}{\pi L}\big(a\cos(kx/L)\cos(ky/L) + a\sin(kx/L)\sin(ky/L)\big)
    \end{equation*}
    with $k\in\N$ and $a>0$.
    \Cref{thm:lowerbound} then yields the lower bound
    \begin{equation}\label{eq:lowerboundsimple}
        t_\rel(\hat P^{(\sigma)})\geq  \Big(\frac{\pi L^3}{2ak^2}\Big)^{1/2}.
    \end{equation}
    on the relaxation time.
    Regarding the upper bound on the relaxation time, a direct computation shows that
    \begin{equation*}
        \frac{1}{(\pi L)^2}\int_{0}^{2\pi L}\cos(kx/L)^j\sin(kx/L)^{4-j}\diff x = \begin{cases}
            \frac{1}{4\pi L}&\text{if }j=2\,,\\
            \frac{3}{4\pi L}&\text{if }j\in\{0,4\}\,,\\
            0&\text{if }j\in\{1,3\}\,,
        \end{cases}
    \end{equation*}
    so that $\tilde\chi = \chi = \sqrt{2\cdot (\frac{3}{4\pi L})^2 + 8\cdot (\frac{1}{4\pi L})^2} = \frac{\sqrt{26}}{4\pi L}$. This yields
    \begin{align*}
        C_1^2 &= 16\pi L\cdot\frac{7\sqrt{26}}{4\pi L} = 28\sqrt{26}<143\qquad\text{and}\qquad C_2^2 = \frac{8\pi L}{a}\,,
    \end{align*}
    so that \Cref{thm:upperbound} yields
    \begin{equation}\label{eq:upperboundsimple}
        t_\rel(\hat P^{(\sigma)})\ \leq\ C\Big(\sigma^2\frac{L}{a} + \sigma^{-2}(1+L^2)\Big)\,.
    \end{equation}
    Choosing a diffusivity minimising this upper bound yields
    \begin{equation*}
        t_\rel(\hat P^{(\sigma_*)})\leq C\sqrt{(L+L^3)/a}\qquad\text{with}\quad\sigma^2_*\propto \sqrt{a/L+aL}\,.
    \end{equation*}
    Comparing this to the lower bound \eqref{eq:lowerboundsimple}, we see that, for this choice of $\sigma$, the upper bound is sharp up to a factor of a constant times $k$ in case $L\geq 1$. 
    
    In comparison, \cite[Theorem 6]{BenaimGauthier2016Selfrepelling} yields the upper bound\footnote{The statement of \cite[Theorem 6]{BenaimGauthier2016Selfrepelling} contains an error in the claimed rate $\lambda$. A factor of $\eta$ (the spectral gap, or $\eta_1$ in the notation of that paper) is missing in $K_1$. Note that it is still present in the proof, specifically Proposition 5.}
    \begin{equation*}
        t_\rel(\hat P^{(\sigma)})\ \leq \ C\,L^2\frac{(1+a\lambda)^2}{a^2\lambda^2}\left(\sigma^2\lambda^2 + \lambda\Big(1+\frac{\sqrt{1+a\lambda}}{L}\Big) + \sigma^{-2} \Big(1+\frac{\sqrt{1+a\lambda}}{L}\Big)^2\right)\,,
    \end{equation*}
    where $\lambda=\frac{k^2}{L^2}$. This upper bound is much weaker than the one in \eqref{eq:upperboundsimple} obtained here, both in the limiting regimes $L\to0$ and $L\to\infty$, as well as when choosing a minimising diffusivity. In particular, it cannot achieve the lower bound \eqref{eq:lowerboundsimple}, since when choosing a minimising diffusivity, we have
    \begin{equation*}
        t_\rel(\hat P^{(\sigma)})\ \leq\ C\,L^2\frac{(1+a\lambda)^2}{a^2\lambda^2}\lambda\Big(1+\frac{\sqrt{1+a\lambda}}{L}\Big)\qquad\text{with}\quad\sigma_*^2\propto\lambda^{-1}\Big(1+\frac{\sqrt{1+a\lambda}}{L}\Big)\,.
    \end{equation*}
    Fixing $a$ and $k$, this yields a relaxation time of the order $O(L^4)$ as $L\to\infty$ and the order $O(L^{-2})$ as $L\to 0$.

\end{exam}

\begin{exam}\label{ex:torus2}
    Let us further investigate how the bounds on the relaxation time depend on the parameter $n$ in \eqref{eq:interactionW}. Note that we have $\chi_{ijkl}\not=0$
    if and only if $(i,j,k,l)$ admits a partition into two tuples (not necessarily of the same size) of equal sum, and the same holds for $\tilde\chi_{ijkl}$. Therefore, there exists a universal constant $C$ such that
    \begin{equation*}
        \chi\leq Cn^{3/2}/L\qquad\text{and}\qquad \tilde\chi\leq Cn^3/L\,.
    \end{equation*}
    This yields
    \begin{equation*}
        C_1^2 \leq C\frac{n^3\sum_{k=1}^na_kk^2}{\min\bigl\{a_kk^2:k\in\{1,\dots,n\}\bigr\}}\quad\text{and}\quad C_2^2\leq C\frac{L}{\min\bigl\{a_k:k\in\{1,\dots,n\}\bigr\}}\,,
    \end{equation*}
    and we observe that the dependence on $L$ does not change compared to the simpler case of \Cref{ex:torus1}.
    In particular, if we choose $a_k=1$ for all $k\in \{1,\dots,n\}$, we obtain an upper bound 
    \begin{equation*}
        t_\rel(\hat P^{(\sigma)}) \leq C \left(\sigma^2 L + \sigma^{-2}(n^6+L^2)\right)
    \end{equation*}
    on the relaxation time.
    Similarly, with $a_k=\frac{1}{k^2}$ for all $k\in \{1,\dots,n\}$, we have
    \begin{equation*}
        t_\rel(\hat P^{(\sigma)}) \leq C \left(\sigma^2 Ln^2 + \sigma^{-2}(n^4+L^2)\right).
    \end{equation*}

    Again, in comparison \cite[Theorem 6]{BenaimGauthier2016Selfrepelling} yields the upper bound
    \begin{equation*}
        t_\rel(\hat P^{(\sigma)})\ \leq \ C\,L^2\frac{(1+\Lambda_0)^2}{\Lambda_0^2}\left(\sigma^2\Lambda_1^2 + \Lambda_1\Big(n+\frac{n^3\sqrt{1+\Lambda_1}}{L}\Big) + \sigma^{-2} \Big(n+\frac{n^3\sqrt{1+\Lambda_1}}{L}\Big)^2\right)\,,
    \end{equation*}
    where $\Lambda_0 = \min\bigl\{a_k\frac{k^2}{L^2}: k\in\{1,\dots,n\}\bigr\}$ and $\Lambda_1=\sum_{k=1}^na_k\frac{k^2}{L^2}$. In the two cases mentioned above, we have 
    \begin{align*}
        \Lambda_0 &= \frac{1}{L^2}\,,\qquad\Lambda_1 \leq C\frac{n^3}{L^2}&\text{if }&a_k=1 \text{ for all } k\in \{1,\dots,n\}\,,\\
        \Lambda_0 &= \frac{1}{L^2}\,,\qquad\Lambda_1 =\frac{n}{L^2}&\text{if }&a_k=\frac{1}{k^2} \text{ for all } k\in \{1,\dots,n\}\,,
    \end{align*}
    yielding worse dependence on $n$ in both cases.
\end{exam}

\section{Convergence to equilibrium}\label{sec:convergence}

We consider the family $(\hat P_t^{(\sigma)})_{t\geq 0}$ of semigroups on $L^2(\hat\mu)$ associated to the self-repelling diffusion \eqref{eq:XtUtSDE} with diffusivity $\sigma\geq 0$. Recall that their generators admit the decomposition
\begin{equation*}
    \hatLgamma f\ =\ \hatL f\ +\ \frac{\sigma^2}{2}\Delta_Mf\qquad\text{for all }f\in\dom(\hatLgamma)
\end{equation*}
given in \eqref{eq:generator}. Convergence to equilibrium of these semigroups can be shown using a flow Poincaré inequality, which can be established for second-order lifts using the framework of \cite{EberleLoerler2024Lifts,EGHLM2025Convergence}. In our setting, the general framework reduces to \Cref{thm:convergencegeneral}. Recall that $(\E,\dom(\E))$ denotes the Dirichlet form \eqref{eq:DiriForm} associated to $\L$, and let
\begin{equation*}
    \Pi f(u,x)\ =\ \frac{1}{\nu(M)}\int_Mf(u,y)\nu(\diff y)\ =\ \int_Mf(u,y)\kappa(\diff y)\,.
\end{equation*}

\begin{theo}\label{thm:convergencegeneral}
    Assume that
    \begin{enumerate}[(i)]
        \item the operator $(\hatL,\dom(\hatL))$ is a second-order lift of $(\L,\dom(\L))$ such that $g\circ\pi\in\dom(\hatL^*)$ and $\hatL^*(g\circ\pi) = -\hatL(g\circ\pi)$ for all $g\in\dom(\L)$;

        \item the operator $(\L,\dom(\L))$ has purely discrete spectrum on $L^2(\mu)$ and a spectral gap $m>0$, i.e.\ 
        \begin{equation*}
            \norm{g-\mu(g)}_{L^2(\mu)}^2\ \leq \ \frac{1}{m}\E(g)
        \end{equation*}
        for all $f\in\dom(\E)$;

        \item there exist constants $C_1,C_2>0$ such that
        \begin{align}
            \big\langle \hatL(g\circ\pi),\hatL(f-\Pi f)\big\rangle_{L^2(\hat\mu)}\ &\leq\ C_1\norm{\L g}_{L^2(\mu)}\norm{\nabla_Mf}_{L^2(\hat\mu)}\,,\label{eq:C1}\\
            \big\langle \hatL(g\circ\pi),\Delta_M f\big\rangle_{L^2(\hat\mu)}\ &\leq\ C_2\norm{\L g}_{L^2(\mu)}\norm{\nabla_Mf}_{L^2(\hat\mu)}\label{eq:C2}
        \end{align}
        for all $f\in\dom(\hatL)$ with $\Pi f\in\dom(\hatL)$ and $g$ in a core of $\L$;

        \item the Laplace-Beltrami operator $\Delta_M$ on $M$ has a spectral gap
        \begin{equation*}
            \eta = \inf\spec(-\Delta_M)\setminus\{0\}>0\,.
        \end{equation*}
    \end{enumerate}
    Then there exists a universal constant $C > 0$ such that, for any $T>0$ and $\sigma>0$, 
    the semigroup $(\hat P_t^{(\sigma)})$ generated by $(\hatLgamma,\dom(\hatLgamma))$ satisfies
    \begin{equation*}
        \lVert\hat P_t^{(\sigma)}f\lVert_{L^2(\hat\mu)}\ \leq\ e^{-\nu(t-T)}\lVert f\rVert_{L^2(\hat\mu)}\qquad\text{for all }t\geq 0\text{ and }f\in L_0^2(\hat\mu),
    \end{equation*}
    where the rate $\nu$ is given by 
    \begin{equation}\label{eq:rategeneral}
        \nu \ =\ C\, \frac{\sigma^2}{\sigma^4C_2^2 + C_1^2 + \frac{1}{\eta}(1+\frac{1}{mT^2})}\,.
    \end{equation}
\end{theo}
\begin{proof}
    The statement is a direct consequence of \cite[Theorem 12]{EGHLM2025Convergence}.
\end{proof}

We will now apply this result to prove the upper bound in \Cref{thm:upperbound} on the relaxation time of self-repelling diffusions.

\begin{lemm}\label{lem:nablaUnablagbound}
    For any $g\in C_\c^\infty(\R^n)$,
    \begin{align*}
        \int_{\R^n}|\nabla\Phi|^2|\nabla g|^2\diff\mu\ \leq\ 2\left(\sum_{k=1}^na_k|\lambda_k|\right)\int_{\R^n}|\nabla g|^2\diff\mu\ +\ 4\int_{\R^n}\|\nabla^2g\|_F^2\diff\mu\,,
    \end{align*}
    where $\norm{\nabla^2g}_F$ denotes the Frobenius norm of $\nabla^2g$.
\end{lemm}
\begin{proof}
    Using integration by parts, we estimate
    \begin{align*}
        \int_{\R^n}|\nabla\Phi|^2|\nabla g|^2\diff\mu &= \int_{\R^n}\Delta\Phi|\nabla g|^2\diff\mu + 2\int_{\R^d}\nabla\Phi^\top\nabla^2g\nabla g\diff\mu\,.
    \end{align*}
    Now,
    \[ 2\nabla\Phi^\top\nabla ^2 g\nabla g\ \leq\ 2|\nabla\Phi|\, |\nabla g|\, \|\nabla^2g\|_F\ \leq\ p|\nabla\Phi|^2|\nabla g|^2 + \frac{1}{p}\|\nabla ^2g\|_F^2 \]
    for any $p>0$ by Young's inequality.
    Choosing $p=\frac{1}{2}$ and rearranging finishes the proof.
\end{proof}
The inequality in \Cref{lem:nablaUnablagbound} can in fact immediately be sharpened by replacing the Frobenius norm with the spectral norm of $\nabla^2g$. The Frobenius norm is, however, more conveniently bounded by Bochner's identity which we will use below.

\begin{lemm}\label{lem:Lhatantisym}
    The operator $(\hatL,C_\c^\infty(\R^n\times M))$ is antisymmetric, i.e.\
    \begin{equation*}
        \int_{\R^n\times M}f\cdot \hatL g\diff\hat\mu = -\int_{\R^n\times M}\hatL f \cdot g\diff\mu
    \end{equation*}
    for any $f,g\in C_\c^\infty(\R^n\times M)$.
\end{lemm}
\begin{proof}
    Note that the adjoint of $\partial_{u_k}$ with respect to $\mu$ is given by $\partial_{u_k}^*f(u) = -\partial_{u_k}f(u)+a_k|\lambda_k|u_kf(u)$.
    An integration by parts thus yields
    \begin{align*}
        \MoveEqLeft\int_{\R^n\times M}f\cdot \hatL g\diff\hat\mu = \int_{\R^n\times M}f(u,x)\sum_{k=1}^n\big(e_k(x)\partial_{u_k}g(u,x) - a_ku_k\langle\nabla_Me_k(x),\nabla_Mg(u,x)\rangle\big)\diff\hat\mu\\
        &=\sum_{k=1}^n\int_{\R^n\times M}\bigl(-\partial_{u_k}f(u,k)e_k(x)+a_k|\lambda_k|u_k f(u,x)g(u,x)\big)\diff\hat\mu\\
        &\qquad+\sum_{k=1}^n\int_{\R^n\times M}a_ku_kg(u,x)\nabla_M\cdot (\nabla_M e_k(x)f(u,x))\diff\hat\mu\\
        &=\sum_{k=1}^n\int_{\R^n\times M}\bigl(-\partial_{u_k}f(u,k)e_k(x)+a_k|\lambda_k|u_k f(u,x)g(u,x)\big)\diff\hat\mu\\
        &\qquad+\sum_{k=1}^n\int_{\R^n\times M}a_ku_kg(u,x)\big(\Delta_Me_k(x)f(u,x) + \langle\nabla_Me_k(x),\nabla_Mf(u,x)\rangle\big)\diff\hat\mu\\
        &=-\int_{\R^n\times M}\hatL f\cdot g\diff\hat\mu\,.
    \end{align*}
\end{proof}

\begin{proof}[Proof of \Cref{thm:upperbound}]
    The claim follows from \Cref{thm:convergencegeneral} if we can show that conditions (i)--(iv) are satisfied with $m$, $C_1$, and $C_2$ as claimed.
    To this end, condition (i) is satisfied by \Cref{thm:lift} and \Cref{lem:Lhatantisym}, and condition (ii) is satisfied with
    \begin{equation*}
        m = \frac{\min\bigl\{a_k|\lambda_k| : k\in\{1,\dots,n\}\bigr\}}{2\nu(M)}
    \end{equation*}
    by \eqref{eq:OUgap}. We turn to condition (iii) and the constant $C_1$. For $g\in C_\c^\infty(\R^n)$ we have $\hatL(g\circ\pi)\in\dom(\hatL^*)$ and
    \begin{align*}
        \hatL^*\hatL(g\circ\pi)(u,x) = -e(x)^\top \nabla ^2g(u)e(x)+\nabla\Phi(u)^\top\Lambda^{-1}\langle\nabla_Me(x),\nabla_Me(x)^\top\rangle\nabla g(u)\,,
    \end{align*}
    where the term $\langle\nabla_Me(x),\nabla_Me(x)^\top\rangle$ is to be understood as
    \begin{equation*}
        \langle\nabla_Me(x),\nabla_Me(x)^\top\rangle = (\langle\nabla_Me_k(x),\nabla_Me_l(x)\rangle)_{k,l=1}^n\in\R^{n\times n}\,.
    \end{equation*}
    Then
    \begin{align*}
        \MoveEqLeft\int_M(\hatL^*\hatL(g\circ\pi)(u,x))^2\kappa(\diff x)\\
        &\leq2\int_M\left(e(x)^\top\nabla^2g(u)e(x)\right)^2+\left(\nabla\Phi(u)^\top\Lambda^{-1}\langle\nabla_Me(x),\nabla_Me(x)^\top\rangle\nabla g(u)\right)^2\kappa(\diff x)\\
        & = \frac{2}{\nu(M)}\sum_{i,j,k,l=1}^n\partial_{ij}g(u)\partial_{kl}g(x)\chi_{ijkl} + \frac{2}{\nu(M)}\sum_{i,j,k,l=1}^n\partial_{i}\Phi(u)\partial_jg(u)\partial_{k}\Phi(u)\partial_lg(u)\tilde\chi_{ijkl}\\
        &\leq \frac{2}{\nu(M)}\left(\norm{\nabla^2g(u)}_F^2\chi + |\nabla\Phi(u)|^2|\nabla g(u)|^2\tilde\chi\right)
    \end{align*}
    By \Cref{lem:nablaUnablagbound}, we thus obtain
    \begin{align*}
        \MoveEqLeft\int_{\R^n\times M}(\hatL^*\hatL(g\circ\pi))^2\diff\hat\mu \leq\frac{2\chi}{\nu(M)}\int_{\R^n}\norm{\nabla ^2g}_F^2\diff\mu + \frac{2\tilde\chi}{\nu(M)}\int_{\R^n}|\nabla\Phi|^2|\nabla g|^2\diff\mu\\
        &\leq \frac{2\chi + 8\tilde\chi}{\nu(M)}\int_{\R^n}\norm{\nabla^2g}_F^2\diff\mu + \frac{4\tilde\chi\sum_{k=1}^na_k|\lambda_k|}{\nu(M)}\int_{\R^n}|\nabla g|^2\diff\mu\,.
    \end{align*}
    Bochner's identity yields
    \begin{equation}\label{eq:Bochner}
        \int_{\R^n}(\nabla^*\nabla g)^2\diff\mu = \int_{\R^n}\norm{\nabla ^2g}_F^2\diff\mu + \int_{\R^n}\sum_{k=1}^na_k|\lambda_k|(\partial_{u_k}g)^2\diff\mu\,,
    \end{equation}
    so that, since $\L = -\frac{1}{2\nu(M)}\nabla^*\nabla$ by \eqref{eq:OUgenerator},
    \begin{equation*}
        \int_{\R^n\times M}(\hatL^*\hatL(g\circ\pi))^2\diff\hat\mu \leq C_1^2\int_{\R^n}(\L g)^2\diff\mu
    \end{equation*}
    with
    \begin{equation*}
        C_1^2 = 8\nu(M)\left(\chi + 4\tilde\chi + \frac{2\tilde\chi\sum_{k=1}^na_k|\lambda_k|}{\min\bigl\{a_k|\lambda_k| : k\in\{1,\dots,n\}\bigr\}}\right)\,.
    \end{equation*}
    Therefore, \eqref{eq:C1} is satisfied with this choice of $C_1$ by the Cauchy-Schwarz inequality.

    Turning to the constant $C_2$, integration by parts, the Cauchy-Schwarz inequality and orthonormality of the $e_k$ yields
    \begin{align*}
        \MoveEqLeft\langle\hatL(g\circ\pi),\Delta_Mf\rangle_{L^2(\hat\mu)} = \int_{\R^n\times M}\sum_{k=1}^ne_k(x)\partial_{u_k}g(u)\Delta_Mf(u,x)\,\hat\mu(\diff u\diff x)\\
        &=\int_{\R^n\times M}\Big\langle\sum_{k=1}^n\partial_{u_k}g(u)\nabla_Me_k(x),\nabla_Mf(u,x)\Big\rangle\,\hat\mu(\diff u\diff x)\\
        &\leq \Big(\int_{\R^n\times M}\Big|\sum_{k=1}^n\partial_{u_k}g(u)\nabla_Me_k(x)\Big|^2\hat\mu(\diff u\diff x)\Big)^{1/2}\Big(\int_{\R^n\times M}|\nabla_Mf|^2\diff\hat\mu\Big)^{1/2}\\
        &=\frac{1}{\sqrt{\nu(M)}}\Big(\int_{\R^n}\sum_{k=1}^n|\lambda_k|(\partial_{u_k}g)^2\diff\mu\Big)^{1/2}\norm{\nabla_Mf}_{L^2(\hat\mu)}\,.
    \end{align*}
    By Bochner's identity \eqref{eq:Bochner} we thus obtain
    \begin{equation*}
        \langle\hatL(g\circ\pi),\Delta_Mf\rangle_{L^2(\hat\mu)}  \leq C_2\norm{\L g}_{L^2(\mu)}\norm{\nabla_Mf}_{L^2(\hat\mu)}
    \end{equation*}
    with
    \begin{equation*}
        C_2^2 = \frac{4\nu(M)}{\min\bigl\{a_k:k\in\{1,\dots,n\}\bigr\}}\,.
    \end{equation*}

\end{proof}

\section*{Acknowledgements}
The author was funded by the Deutsche Forschungsgemeinschaft (DFG, German Research Foundation) under Germany’s Excellence Strategy EXC 2047 -- 390685813 as well as under CRC 1720 -- 539309657. I would like to thank Andreas Eberle, Leo Hahn, and Michel Benaïm for helpful discussions and comments, and Michel Benaïm for funding a stay in Neuch\^atel in October 2025.

\printbibliography

\end{document}